\documentclass[12pt,a4paper]{amsart}

\usepackage{amsmath,amssymb,amsthm,mathrsfs}
\usepackage[utf8]{inputenc}
\usepackage[T1]{fontenc}
\usepackage{hyperref}
\usepackage{cleveref}
\usepackage{enumitem}
\usepackage{appendix}

\theoremstyle{plain}
\newtheorem{theorem}{Theorem}[section]
\newtheorem{proposition}[theorem]{Proposition}
\newtheorem{lemma}[theorem]{Lemma}
\newtheorem{corollary}[theorem]{Corollary}

\theoremstyle{definition}
\newtheorem{definition}[theorem]{Definition}
\newtheorem{example}[theorem]{Example}

\theoremstyle{remark}
\newtheorem{remark}[theorem]{Remark}

\newcommand{\kk}{k}
\newcommand{\eps}{\epsilon}
\newcommand{\ord}{\operatorname{ord}}
\newcommand{\Spec}{\operatorname{Spec}}
\newcommand{\Hilb}{\operatorname{Hilb}}
\newcommand{\CHilb}{\operatorname{CHilb}}

\newcommand{\N}{\mathbb{N}}
\newcommand{\C}{\mathbb{C}}

\title[Explicit deformation of a spider algebra]{Explicit deformation of a spider algebra\\
to a curvilinear scheme via M\"obius generators}

\author{David Turturean}
\address{Massachusetts Institute of Technology, Cambridge, MA 02139, USA}
\email{davidct@mit.edu}


\begin{document}

\begin{abstract}
We construct an explicit flat one-parameter family of 22-dimensional Artinian
$\kk$-algebras whose special fibre is the spider algebra
$\kk[x,y,z]/(x^8,y^8,z^8,xy,xz,yz)$ and whose generic fibre is the
curvilinear algebra $\kk[t]/(t^{22})$.  The construction uses
M\"obius generators $u_a = t/(1-at)$ inside the curvilinear ring together
with a divided-difference change of coordinates, and produces the family
via a weighted Rees degeneration with integer coefficients.  This gives an explicit one-parameter family witnessing, for this spider
ideal, the general phenomenon proved by B\'erczi--Svendsen that every
monomial subscheme of~$\C^d$ lies in the curvilinear component of the
Hilbert scheme of points.
\end{abstract}

\maketitle
\setcounter{tocdepth}{1}
\tableofcontents

%% ══════════════════════════════════════════════════════════════════
\section{Introduction}\label{sec:intro}
%% ══════════════════════════════════════════════════════════════════

Let $\kk$ be a field of characteristic zero.  The \emph{Hilbert scheme of
$n$~points} $\Hilb^n(\C^d)$ parametrises zero-dimensional subschemes of
length~$n$ in affine $d$-space.  The \emph{curvilinear component}
$\CHilb^n_0(\C^d)$ is the closure of the locus of schemes isomorphic to
$\Spec\,\kk[t]/(t^n)$; it is irreducible and generically smooth.

A classical theorem of Brian\c{c}on~\cite{Briancon77} establishes that for
$d=2$, every monomial ideal of colength~$n$ lies in the curvilinear
component.  The general case was settled recently by B\'erczi and
Svendsen~\cite{BercziSvendsen23}:

\begin{theorem}[{\cite[Theorem~1]{BercziSvendsen23}}]\label{thm:BS}
Every monomial ideal in\/ $\Hilb^n(\C^d)$ lies in the curvilinear component
$\CHilb^n_0(\C^d)$.
\end{theorem}

The proof in~\cite{BercziSvendsen23} proceeds via the test curve model and
is existential: it does not produce explicit deformation families.  Moreover,
the argument relies on socle extensions that increase the ambient dimension,
and the authors note that earlier fixed-dimension approaches had been
unsuccessful.
Constructing explicit families in the original ambient dimension is of
independent interest for several reasons.
Explicit families give local coordinates on the Hilbert scheme near singular
points, enable effective computations in deformation theory, and are required
for equivariant localisation calculations in enumerative
geometry~\cite{Berczi23,Berczi21}.

\medskip\noindent\textbf{Related work.}\enspace
Prior explicit constructions of flat families from monomial ideals have been
limited in scope.
Cartwright--Erman--Velasco--Viray~\cite{CEVV09} classified all smoothable
ideals of colength at most~8 by exhibiting, for each monomial ideal~$I$, a
weight vector~$w$ and a smoothable ideal~$J$ with $I=\mathrm{in}_w(J)$, so
that the Gr\"obner degeneration provides the flat family
(see~\cite[Section~4]{CEVV09}).  Their method targets smoothability
rather than curvilinearity specifically, and does not extend beyond
colength~8.
In a complementary direction, Jelisiejew--Ramkumar--Sammartano~\cite{JRS24}
introduced \emph{broken Gorenstein structures} and gave an explicit
combinatorial characterisation of which monomial points on\/
$\Hilb^d(\mathbb{A}^3)$ are smooth
(see~\cite[Theorem~4.8]{JRS24}); their framework addresses singularity
classification rather than explicit one-parameter families.
On the foundational side, Fogarty~\cite{Fogarty68} proved that\/
$\Hilb^n(S)$ is irreducible for smooth surfaces~$S$ (so every ideal
is smoothable in two variables), while
Iarrobino~\cite{Iarrobino72} showed that\/ $\Hilb^d(\mathbb{A}^n)$ is
reducible for $n\geq 3$ and sufficiently large~$d$, making the
curvilinear component a proper subvariety.
In the analytic setting, Brian\c{c}on--Galligo~\cite{BrianconGalligo73}
constructed explicit flat deformations of zero-dimensional ideals in
$\kk\{x,y\}$ via standard bases and Weierstrass preparation; their work
is the closest classical precedent for ``writing down generators and proving
flatness,'' though it is restricted to two variables and targets reduced
(rather than curvilinear) fibres.
Our construction appears to be the first explicit flat family deforming a
monomial ideal of colength greater than~8 in three or more variables to a
curvilinear algebra.

\medskip\noindent\textbf{Motivation and methods.}\enspace
This work was motivated by a problem posed in the Epoch~AI
\emph{FrontierMath: Open Problems} collection~\cite{FrontierMath},
which asked for an explicit flat deformation
from the curvilinear algebra $\kk[t]/(t^{22})$ to the spider algebra of type
$(7,7,7)$.
AI tools were used substantively in the discovery and verification of the
construction presented here; a detailed account of how AI was employed will
appear in a forthcoming companion preprint.

In this paper we give a fully explicit construction for the
$22$-dimen\-sional \emph{spider algebra}
\[
  A \;=\; \kk[x,y,z]/(x^8,\,y^8,\,z^8,\,xy,\,xz,\,yz),
\]
which corresponds to a monomial ideal in three variables with three legs of
length~7.  Our main result is:

\begin{theorem}[Main Theorem]\label{thm:main}
There exists an explicit ideal
$I\subset\kk[x,y,z,\eps]$ with integer coefficients such that the family
\[
  \mathcal{A} \;=\; \kk[x,y,z,\eps]/I
\]
is flat over $\kk[\eps]$, with
\begin{enumerate}[label=\textup{(\roman*)}]
\item $\mathcal{A}/\eps\,\mathcal{A} \;\cong\; A$,
\item $\mathcal{A}[\eps^{-1}] \;\cong\;
        \kk[t]/(t^{22})\otimes_\kk\kk[\eps,\eps^{-1}]$.
\end{enumerate}
\end{theorem}

\medskip\noindent\textbf{Method.}\enspace
The idea is to embed three ``legs'' of the spider into the curvilinear ring
$R=\kk[t]/(t^{22})$ via M\"obius generators
$u_a=t/(1-at)$ for $a=1,2,3$, pass to divided-difference coordinates
$(x,y,z)$, prove these give a basis, derive the generic relations, and then
form the weighted Rees degeneration.
The M\"obius generators are a natural choice: the rational functions
$t\mapsto t/(1-at)$ have well-separated poles at $t=1/a$, giving clean
multiplication rules via the identity $u_b-u_a=(b-a)u_au_b$ and making
linear independence transparent through a pole-order argument.
This is in the spirit of classical explicit-deformation constructions
(cf.~\cite{BrianconGalligo73}), where control of generators and syzygies
is the central technical challenge.

\medskip\noindent\textbf{Outline.}\enspace
\Cref{sec:prelim} fixes notation and recalls background.
\Cref{sec:warmup} treats the illustrative case
$\kk[t]/(t^3)\leadsto\kk[x,y]/(x,y)^2$.
\Cref{sec:mobius} introduces the M\"obius generators and proves the basis
theorem.
\Cref{sec:generic} derives all relations in the generic ideal, with
complete proofs.
\Cref{sec:rees} constructs the weighted Rees family.
\Cref{sec:flat} proves flatness.
\Cref{sec:comp} describes supplementary computational verifications.
\Cref{sec:remarks} discusses generalisations.
\Cref{app:purepower} contains the derivation and verification of the
pure-power relations $g_x$ and $g_y$.

%% ══════════════════════════════════════════════════════════════════
\section{Preliminaries}\label{sec:prelim}
%% ══════════════════════════════════════════════════════════════════

Throughout, $\kk$ denotes a field of characteristic zero.

\begin{definition}\label{def:spider}
Let $\ell_1,\ldots,\ell_r\ge 1$ be positive integers.  The
\emph{spider algebra} (or \emph{star-shaped monomial algebra}) of type
$(\ell_1,\ldots,\ell_r)$ is
\[
  A \;=\; \kk[x_1,\ldots,x_r]/
          (x_i^{\ell_i+1},\; x_ix_j \text{ for } i\neq j).
\]
Its $\kk$-dimension is $1+\sum_{i=1}^r\ell_i$.
The corresponding monomial ideal defines a fat point in~$\C^r$ whose
support diagram has $r$~legs of lengths $\ell_1,\ldots,\ell_r$.
\end{definition}

In our case $r=3$ and $\ell_1=\ell_2=\ell_3=7$, giving
$\dim_\kk A = 1+7+7+7=22$.

\begin{definition}\label{def:curvilinear}
The \emph{curvilinear algebra} of length~$n$ is $R=\kk[t]/(t^n)$.  A
zero-dimensional scheme is called \emph{curvilinear} if its local ring is
isomorphic to such an algebra.
\end{definition}

\begin{definition}\label{def:flatfam}
A \emph{flat one-parameter deformation} of a $\kk$-algebra~$A$ to a
$\kk$-algebra~$B$ is a finitely generated $\kk[\eps]$-algebra $\mathcal{A}$
which is flat as a $\kk[\eps]$-module, together with isomorphisms
$\mathcal{A}/\eps\mathcal{A}\cong A$ and
$\mathcal{A}\otimes_{\kk[\eps]}\kk(\eps)\cong B\otimes_\kk\kk(\eps)$.
\end{definition}

We recall the standard mechanism:

\begin{definition}\label{def:rees}
Let $S=\kk[x_1,\ldots,x_s]/J$ and let $w\colon\{x_1,\ldots,x_s\}\to\N$
be a weight function, extended to monomials additively.  For a polynomial
$f=\sum_\alpha c_\alpha x^\alpha\in J$, the \emph{weighted Rees
homogenisation} with respect to~$w$ is
\[
  f^w \;=\; \sum_\alpha c_\alpha\,\eps^{\,w_{\max}(f)-w(\alpha)}\,x^\alpha,
\]
where $w_{\max}(f)=\max\{w(\alpha):c_\alpha\neq 0\}$.
If $J=(f_1,\ldots,f_m)$ and if for each~$f_i$ the maximum weight is achieved
by a unique leading monomial, the family
$\kk[x_1,\ldots,x_s,\eps]/(f_1^w,\ldots,f_m^w)$
has the property that:
\begin{enumerate}[label=\textup{(\alph*)}]
\item at $\eps=0$ the fibre is
  $\kk[x_1,\ldots,x_s]/(\operatorname{in}_w(f_1),\ldots,\operatorname{in}_w(f_m))$;
  this equals $\kk[x_1,\ldots,x_s]/\operatorname{in}_w(J)$ when the $f_i$
  form a Gr\"obner basis with respect to the weight order,
\item for $\eps=\lambda\neq 0$, the substitution $x_i=\lambda^{w_i}X_i$
      recovers the original ideal~$J$.
\end{enumerate}
\end{definition}

%% ══════════════════════════════════════════════════════════════════
\section{Warm-up: the 3-point case}\label{sec:warmup}
%% ══════════════════════════════════════════════════════════════════

Before tackling the 22-dimensional problem, we illustrate the method on the
simplest spider deformation:
\[
  \kk[t]/(t^3) \;\leadsto\; \kk[x,y]/(x,y)^2.
\]
Both algebras have $\kk$-dimension~$3$.

\begin{example}\label{ex:warmup}
Inside $R=\kk[t]/(t^3)$, set $x=t$ and $y=t^2$.  The relation
$y=x^2$ holds in~$R$, so the surjection $\kk[x,y]\to R$ has kernel
$(y-x^2)$, giving
\[
  \kk[x,y]/(y-x^2) \;\cong\; R.
\]
Now assign weights $w(x)=1$, $w(y)=2$.  In the target spider ideal
$(x^2,xy,y^2)$, the monomial $x^2$ has weight~$2$ while $y$ also has
weight~$2$; so the generic relation $y=x^2$ has tied leading weight, and
the standard Rees homogenisation is trivial.  Instead, we build the family
directly: the ideal $(x^2-\eps\,y,\;xy,\;y^2)$ interpolates between the
spider at $\eps=0$ and the curvilinear algebra at $\eps\neq 0$.  Concretely,
the resulting family is
\[
  \mathcal{B} \;=\; \kk[x,y,\eps]/(x^2-\eps y,\;xy,\;y^2).
\]
At $\eps=0$ the fibre is $\kk[x,y]/(x^2,xy,y^2)=\kk[x,y]/(x,y)^2$.  For
$\eps\neq 0$, the relation $x^2=\eps y$ makes $y$ redundant, recovering
$\kk[x]/(x^3)\cong\kk[t]/(t^3)$.
\end{example}

This captures the essential mechanism: powers of a single generator bifurcate
into independent variables in the limit, controlled by a weighted
deformation parameter.  The 22-dimensional case requires a more sophisticated
embedding.

%% ══════════════════════════════════════════════════════════════════
\section{M\"obius generators and the basis theorem}\label{sec:mobius}
%% ══════════════════════════════════════════════════════════════════

\subsection{The embedding}\label{ssec:embedding}

Let $R=\kk[t]/(t^{22})$ and work in the formal power series ring
$\kk[[t]]/(t^{22})$.  For each $a\in\kk$, define the \emph{M\"obius
generator}
\[
  u_a \;=\; \frac{t}{1-at}.
\]
Since $1-at$ is a unit in the local ring $\kk[[t]]/(t^{22})$, this is
well-defined.  Note that $\ord_t(u_a)=1$ for all~$a$.

\begin{lemma}[M\"obius identity]\label{lem:mobius}
For all $a,b\in\kk$,
\begin{equation}\label{eq:mobius-identity}
  u_b - u_a \;=\; (b-a)\,u_a\,u_b.
\end{equation}
\end{lemma}

\begin{proof}
Compute directly in the fraction field of $\kk[[t]]$:
\begin{align*}
  u_b - u_a
  &\;=\; \frac{t}{1-bt} - \frac{t}{1-at}
  \;=\; \frac{t(1-at) - t(1-bt)}{(1-at)(1-bt)} \\[4pt]
  &\;=\; \frac{t^2(b-a)}{(1-at)(1-bt)}
  \;=\; (b-a)\cdot\frac{t}{1-at}\cdot\frac{t}{1-bt} \\[4pt]
  &\;=\; (b-a)\,u_a\,u_b.
  \qedhere
\end{align*}
\end{proof}

We use the three generators $u_1,u_2,u_3$ and pass to \emph{divided-difference
coordinates}:
\begin{align}
  x &\;:=\; u_1 \;=\; \frac{t}{1-t},
    \label{eq:defx}\\[4pt]
  y &\;:=\; u_2 - u_1 \;=\; \frac{t^2}{(1-t)(1-2t)},
    \label{eq:defy}\\[4pt]
  z &\;:=\; u_3 - 2u_2 + u_1 \;=\; \frac{2t^3}{(1-t)(1-2t)(1-3t)}.
    \label{eq:defz}
\end{align}

We verify the closed forms.  For~$y$: by \Cref{lem:mobius} with $a=1$, $b=2$,
\[
  y = u_2 - u_1 = (2-1)\,u_1 u_2 = u_1 u_2
    = \frac{t}{1-t}\cdot\frac{t}{1-2t}
    = \frac{t^2}{(1-t)(1-2t)}.
\]
For~$z$: by \Cref{lem:mobius},
$u_3 - u_2 = u_2 u_3$ and $u_2 - u_1 = u_1 u_2$, so
\begin{align*}
  z &= (u_3 - u_2) - (u_2 - u_1) = u_2 u_3 - u_1 u_2
     = u_2(u_3 - u_1).
\end{align*}
Applying \Cref{lem:mobius} with $a=1$, $b=3$: $u_3-u_1 = 2\,u_1u_3$.  Hence
\begin{align*}
  z &= u_2\cdot 2\,u_1u_3
     = 2\,u_1 u_2 u_3
     = \frac{2t^3}{(1-t)(1-2t)(1-3t)}.
\end{align*}
The $t$-adic orders are $\ord_t(x)=1$, $\ord_t(y)=2$, $\ord_t(z)=3$.

\begin{remark}\label{rmk:divideddiff}
The divided-difference structure is visible: writing $\Delta u_a = u_{a+1}-u_a$,
we have $x = u_1$, $y = \Delta u_1$, $z = \Delta^2 u_1$.  The factorisation
$z = 2\,u_1 u_2 u_3$ explains why $\ord_t(z)=3$.
\end{remark}

\subsection{The basis theorem}

Consider the set of \emph{spider monomials}
\begin{equation}\label{eq:basisB}
  B \;=\; \{1,\,x,\ldots,x^7,\;\; y,\ldots,y^7,\;\; z,\ldots,z^7\},
  \qquad |B|=22.
\end{equation}

\begin{theorem}[Basis]\label{thm:basis}
The set~$B$ is a $\kk$-basis of\/ $R=\kk[t]/(t^{22})$.
\end{theorem}

\begin{proof}
Since $\dim_\kk R=22=|B|$, it suffices to show that $B$ is linearly
independent.  Suppose
\[
  a_0 + \sum_{i=1}^{7} a_i\,x^i + \sum_{j=1}^{7} b_j\,y^j
        + \sum_{k=1}^{7} c_k\,z^k \;=\; 0 \quad\text{in } R.
\]
Since $\ord_t(x^i)\ge 1$, $\ord_t(y^j)\ge 2$, and $\ord_t(z^k)\ge 3$,
evaluating at $t=0$ immediately gives $a_0=0$.  It remains to show
\[
  F \;:=\; \sum_{i=1}^{7} a_i\,x^i + \sum_{j=1}^{7} b_j\,y^j
        + \sum_{k=1}^{7} c_k\,z^k \;=\; 0
\]
implies $a_i=b_j=c_k=0$.  We have $F(t)\equiv 0\pmod{t^{22}}$.  Define
\[
  D(t) \;=\; (1-t)^7\,(1-2t)^7\,(1-3t)^7.
\]

\medskip\noindent\textbf{Step 1: $D\cdot F$ is a polynomial of degree
$\le 21$.}\enspace
From the closed forms \eqref{eq:defx}--\eqref{eq:defz}:
\[
  x^i = \frac{t^i}{(1-t)^i},\qquad
  y^j = \frac{t^{2j}}{(1-t)^j(1-2t)^j},\qquad
  z^k = \frac{2^k\,t^{3k}}{\prod_{\ell=1}^{3}(1-\ell t)^k}.
\]
Multiplying by $D(t)$:
\begin{align*}
  D\cdot x^i &= t^i(1\!-\!t)^{7-i}(1\!-\!2t)^7(1\!-\!3t)^7, \\
  D\cdot y^j &= t^{2j}(1\!-\!t)^{7-j}(1\!-\!2t)^{7-j}(1\!-\!3t)^7, \\
  D\cdot z^k &= 2^k t^{3k}(1\!-\!t)^{7-k}(1\!-\!2t)^{7-k}(1\!-\!3t)^{7-k}.
\end{align*}
Since $i,j,k\le 7$, all exponents $7-i$, $7-j$, $7-k$ are non-negative,
so each product is indeed a polynomial of degree exactly~$21$.  Hence
$D(t)F(t)$ is a polynomial of degree at most~$21$.

\medskip\noindent\textbf{Step 2: $D\cdot F\equiv 0$.}\enspace
Since $F(t)=O(t^{22})$, the polynomial $D(t)F(t)$ has its first $22$ Taylor
coefficients equal to zero.  But $\deg(DF)\le 21<22$, so $D(t)F(t)$ is
identically the zero polynomial.

\medskip\noindent\textbf{Step 3: Pole separation.}\enspace
Since $DF\equiv 0$ as a polynomial and $D$ is nonzero, $F\equiv 0$ as a
rational function of~$t$.  We now examine the poles of~$F$ to show all
coefficients vanish.

\emph{Poles at $t=\frac{1}{3}$:}\enspace
Among the terms of $F$, only $z^k$ has a pole at $t=\frac{1}{3}$
(the factor $(1-3t)^k$ in the denominator).  Neither $x^i$ nor $y^j$
has any pole there.  The pole of $z^k$ at $t=\frac{1}{3}$ has order exactly~$k$.
Since $F\equiv 0$ and the poles at $t=\frac{1}{3}$ have distinct
orders $1,2,\ldots,7$, we conclude $c_k=0$ for all $k=1,\ldots,7$.

\emph{Poles at $t=\frac{1}{2}$:}\enspace
With all $c_k=0$, only the $y^j$-terms remain that have poles at
$t=\frac{1}{2}$ (from the factor $(1-2t)^j$).  The $x^i$-terms have no
pole there.  The pole of $y^j$ at $t=\frac{1}{2}$ has order~$j$, and these
are distinct for $j=1,\ldots,7$.  Hence $b_j=0$ for all~$j$.

\emph{Poles at $t=1$:}\enspace
With $b_j=c_k=0$, only the $x^i$-terms remain.  Each $x^i$ has a pole
of order~$i$ at $t=1$, and these are distinct.  Hence $a_i=0$ for all~$i$.

Therefore all coefficients vanish and $B$ is linearly independent, hence a
basis.
\end{proof}

\begin{corollary}\label{cor:curvilinear}
The element $x=t/(1-t)$ generates~$R$ as a $\kk$-algebra, and
$R\cong\kk[x]/(x^{22})\cong\kk[t]/(t^{22})$.
\end{corollary}

\begin{proof}
From $x=t/(1-t)$ we get $t=x/(1+x)$.  Since $\ord_t(x)=1$, the element
$x$ is nilpotent in~$R$ ($x^{22}=0$), so $1+x$ is a unit with inverse
$\sum_{n=0}^{21}(-x)^n$.  Hence $t=x(1+x)^{-1}\in\kk[x]\subset R$, so $\kk[x]=R$, and
$x^{21}\neq 0$ but $x^{22}=0$.
\end{proof}

%% ══════════════════════════════════════════════════════════════════
\section{The generic ideal}\label{sec:generic}
%% ══════════════════════════════════════════════════════════════════

We now derive all relations in the ideal $J$ defining $R$ as a quotient
of $\kk[x,y,z]$.  Every step is proved from the M\"obius identity
(\Cref{lem:mobius}) and the definitions \eqref{eq:defx}--\eqref{eq:defz}.

\subsection{Useful identities}\label{ssec:useful}

Before deriving the relations, we record several identities that will be
used repeatedly.  All follow directly from \Cref{lem:mobius}.

\begin{lemma}\label{lem:products}
The following hold in~$R$:
\begin{enumerate}[label=\textup{(\alph*)}]
\item $u_1 u_2 = y$ \quad(from $u_2-u_1 = u_1 u_2$),\label{it:u1u2}
\item $u_2 u_3 = u_3 - u_2$ \quad(from \Cref{lem:mobius} with $b-a=1$),
  \label{it:u2u3}
\item $u_1 u_3 = \frac{1}{2}(u_3-u_1)$
  \quad(from \Cref{lem:mobius} with $b-a=2$),\label{it:u1u3}
\item $u_2 = x+y$,\label{it:u2}
\item $u_3 - u_1 = z + 2y$
  \quad(since $u_3-u_1 = (u_3-2u_2+u_1)+2(u_2-u_1)=z+2y$),\label{it:u3mu1}
\item $u_3 - u_2 = z + y$
  \quad(since $u_3-u_2 = (u_3-2u_2+u_1)+(u_2-u_1)=z+y$).\label{it:u3mu2}
\end{enumerate}
\end{lemma}

\begin{proof}
\ref{it:u1u2}: By \Cref{lem:mobius} with $a=1$, $b=2$: $u_2-u_1=(2-1)u_1u_2
= u_1u_2$.  Since $y=u_2-u_1$ by definition, $u_1u_2=y$.

\ref{it:u2u3}: By \Cref{lem:mobius} with $a=2$, $b=3$: $u_3-u_2=(3-2)u_2u_3
= u_2u_3$.

\ref{it:u1u3}: By \Cref{lem:mobius} with $a=1$, $b=3$: $u_3-u_1=(3-1)u_1u_3
= 2u_1u_3$, hence $u_1u_3=\frac{1}{2}(u_3-u_1)$.

\ref{it:u2}: $u_2 = u_1+(u_2-u_1) = x+y$.

\ref{it:u3mu1}: $u_3-u_1 = (u_3-2u_2+u_1)+2(u_2-u_1) = z+2y$.

\ref{it:u3mu2}: $u_3-u_2 = (u_3-2u_2+u_1)+(u_2-u_1) = z+y$.
\end{proof}

\subsection{Mixed relations}\label{ssec:mixed}

\begin{proposition}\label{prop:mixed}
The following relations hold in~$R$:
\begin{align}
  xy &\;=\; y - x^2,\label{eq:rel-xy}\\
  2xz &\;=\; z - 2y + 2x^2,\label{eq:rel-xz}\\
  2yz &\;=\; z - 2y + 2x^2 - 4y^2.\label{eq:rel-yz}
\end{align}
\end{proposition}

\begin{proof}
We derive each relation in full.

\medskip\noindent\textbf{Relation \eqref{eq:rel-xy}: $xy = y-x^2$.}\enspace
Recall $x = u_1$ and $y = u_2 - u_1$.  By \Cref{lem:products}\ref{it:u1u2},
$y = u_1 u_2 = x\,u_2$, so $u_2 = x + y$ (\Cref{lem:products}\ref{it:u2}).
Therefore:
\[
  xy \;=\; x(u_2 - u_1) \;=\; x\,u_2 - x^2 \;=\; y - x^2.
\]

\medskip\noindent\textbf{Relation \eqref{eq:rel-xz}: $2xz = z-2y+2x^2$.}\enspace
We expand $xz = u_1(u_3 - 2u_2 + u_1)$:
\[
  xz \;=\; u_1 u_3 - 2u_1 u_2 + u_1^2.
\]
Substituting the identities from \Cref{lem:products}:
\begin{align*}
  u_1 u_3 &= \tfrac{1}{2}(u_3-u_1) = \tfrac{1}{2}(z+2y) = \tfrac{z}{2}+y
    \quad\text{(using \ref{it:u1u3} and \ref{it:u3mu1})}, \\
  u_1 u_2 &= y
    \quad\text{(using \ref{it:u1u2})}, \\
  u_1^2 &= x^2.
\end{align*}
Hence:
\[
  xz \;=\; \bigl(\tfrac{z}{2}+y\bigr) - 2y + x^2
     \;=\; \tfrac{z}{2} - y + x^2.
\]
Multiplying by~$2$:
\[
  2xz \;=\; z - 2y + 2x^2.
\]

\medskip\noindent\textbf{Relation \eqref{eq:rel-yz}: $2yz = z-2y+2x^2-4y^2$.}\enspace
We expand $yz = (u_2-u_1)(u_3-2u_2+u_1)$ completely:
\begin{equation}\label{eq:yz-expand}
  yz \;=\; u_2 u_3 - 2u_2^2 + u_1 u_2 - u_1 u_3 + 2u_1 u_2 - u_1^2.
\end{equation}
We evaluate each of the six terms.

\emph{Term $u_2u_3$:}\enspace By \Cref{lem:products}\ref{it:u2u3} and
\ref{it:u3mu2}: $u_2u_3 = u_3-u_2 = z+y$.

\emph{Term $u_1u_2$ and $2u_1u_2$:}\enspace By
\Cref{lem:products}\ref{it:u1u2}: $u_1u_2=y$, so $u_1u_2+2u_1u_2 = 3y$.

\emph{Term $u_1u_3$:}\enspace By \Cref{lem:products}\ref{it:u1u3} and
\ref{it:u3mu1}: $u_1u_3 = \frac{1}{2}(z+2y) = \frac{z}{2}+y$.

\emph{Term $u_1^2$:}\enspace $u_1^2 = x^2$.

\emph{Term $u_2^2$:}\enspace Using $u_2=x+y$
(\Cref{lem:products}\ref{it:u2}) and the already-proved relation
$xy = y-x^2$ \eqref{eq:rel-xy}:
\begin{equation}\label{eq:u2sq}
  u_2^2 = (x+y)^2 = x^2+2xy+y^2 = x^2+2(y-x^2)+y^2 = -x^2+2y+y^2.
\end{equation}

Substituting everything into~\eqref{eq:yz-expand}:
\begin{align*}
  yz &= (z+y) - 2(-x^2+2y+y^2) + 3y - (\tfrac{z}{2}+y) - x^2 \\
     &= z+y+2x^2-4y-2y^2+3y-\tfrac{z}{2}-y-x^2 \\
     &= \tfrac{z}{2}+x^2-y-2y^2.
\end{align*}
Multiplying by~$2$:
\[
  2yz \;=\; z + 2x^2 - 2y - 4y^2 \;=\; z - 2y + 2x^2 - 4y^2.
  \qedhere
\]
\end{proof}

\subsection{Pure-power relations}\label{ssec:pure}

Since $\ord_t(z)=3$, we have $\ord_t(z^8)=24>21$, hence
\begin{equation}\label{eq:z8}
  z^8 = 0 \quad\text{in } R = \kk[t]/(t^{22}).
\end{equation}

For $x^8$ and $y^8$, the situation is more involved: $\ord_t(x^8)=8$ and
$\ord_t(y^8)=16$, both $\le 21$, so these are nonzero in~$R$.  Since $B$
is a basis (\Cref{thm:basis}), $x^8$ and $y^8$ each have a unique
$B$-expansion.  Clearing the denominators that arise from the rational
expressions of the basis elements gives integer-coefficient polynomial
relations $g_x = 0$ and $g_y = 0$.

The derivation method and the explicit verification that the claimed
relations hold are given in \Cref{app:purepower}.

\begin{proposition}\label{prop:pure}
The following relations hold in~$R$:
\begin{equation}\label{eq:gx}
\begin{split}
  g_x \;:=\; 32\,x^8
  &{}-1728\,x^7 - 2864\,x^6 - 11088\,x^5 - 14988\,x^4 \\
  &{}- 28080\,x^3 - 23484\,x^2 + 23484\,y \\
  &{}- 2048\,y^7 - 9728\,y^6 - 18944\,y^5 - 25888\,y^4 \\
  &{}- 20384\,y^3 - 22284\,y^2 + 2298\,z \\
  &{}+ 16\,z^7 - 8\,z^6 + 16\,z^5 - 46\,z^4 + 156\,z^3 - 581\,z^2
  \;=\; 0,
\end{split}
\end{equation}
and
\begin{equation}\label{eq:gy}
\begin{split}
  g_y \;:=\; 2048\,y^8
  &{}-112608\,x^7 - 196272\,x^6 - 723264\,x^5 \\
  &{}- 1000056\,x^4 - 1835964\,x^3 \\
  &{}- 1556406\,x^2 + 1556406\,y \\
  &{}- 116736\,y^7 - 582656\,y^6 - 1144064\,y^5 \\
  &{}- 1576192\,y^4 - 1226864\,y^3 \\
  &{}- 1395024\,y^2 + 139779\,z \\
  &{}+ 1008\,z^7 - 496\,z^6 + 984\,z^5 - 2816\,z^4 \\
  &{}+ 9522\,z^3 - 35392\,z^2
  \;=\; 0.
\end{split}
\end{equation}
\end{proposition}

\subsection{The generic ideal}

Let
\begin{equation}\label{eq:idealJ}
  J \;=\; \bigl(
    xy+x^2-y,\;\;
    2xz-2x^2+2y-z,\;\;
    2yz-2x^2+2y+4y^2-z,\;
    g_x,\, g_y,\, z^8
  \bigr)
\end{equation}
in $\kk[x,y,z]$.

\begin{proposition}\label{prop:generic}
There is an isomorphism $\kk[x,y,z]/J \;\cong\; \kk[t]/(t^{22})$.
\end{proposition}

\begin{proof}
The substitution \eqref{eq:defx}--\eqref{eq:defz} defines a $\kk$-algebra
map $\varphi\colon\kk[x,y,z]\to R$ which is surjective by
\Cref{cor:curvilinear} (since $x$ alone generates~$R$).  By
\Cref{prop:mixed}, \Cref{prop:pure}, and~\eqref{eq:z8}, all generators
of~$J$ lie in $\ker\varphi$.  Hence $\varphi$ descends to a surjection
$\bar\varphi\colon\kk[x,y,z]/J\twoheadrightarrow R$.

It remains to show $\dim_\kk(\kk[x,y,z]/J)\le 22$.

\medskip\noindent\textbf{Step~1: Reduction of mixed monomials.}\enspace
From $xy = y-x^2$ (relation~\eqref{eq:rel-xy}), every occurrence of the
product $xy$ in a monomial can be replaced by $y-x^2$, a sum of pure powers.

From $2xz = z-2y+2x^2$ (relation~\eqref{eq:rel-xz}), every occurrence of
$xz$ is replaced by $\frac{1}{2}(z-2y+2x^2)$, again a sum of pure powers.
(In characteristic zero, dividing by~$2$ is valid.)

From $2yz = z-2y+2x^2-4y^2$ (relation~\eqref{eq:rel-yz}), every occurrence
of $yz$ is replaced by $\frac{1}{2}(z-2y+2x^2-4y^2)$, a sum of pure powers.

By iterating these reductions, every
monomial in $\kk[x,y,z]$ eventually reduces to a $\kk$-linear combination
of pure powers $\{x^i : i\ge 0\}\cup\{y^j : j\ge 1\}\cup\{z^k : k\ge 1\}$.
To see termination, assign weights $w(x)=15$, $w(y)=16$, $w(z)=17$
(extended multiplicatively to monomials).  In each rewrite rule, every
output monomial has strictly smaller weight than the input:
$w(xy)=31>30=w(x^2)$ for~\eqref{eq:rel-xy};
$w(xz)=32>30=w(x^2)$ for~\eqref{eq:rel-xz};
$w(yz)=33>32=w(y^2)$ for~\eqref{eq:rel-yz}
(listing only the heaviest output term in each case).
Multiplying by a common monomial factor preserves strict inequality,
so each application strictly decreases the weight.  Since weights
are non-negative integers, the process terminates.

\medskip\noindent\textbf{Step~2: Reduction of high pure powers.}\enspace
The relation $g_x=0$ expresses $32\,x^8$ as a linear combination of elements
of~$B$ (the leading coefficient $32$ is invertible in characteristic zero).
Similarly, $g_y=0$ expresses $2048\,y^8$ in terms of~$B$, and $z^8=0$
eliminates $z^8$ directly.  By iterating (applying $xy=y-x^2$ etc.\ to
reduce any mixed terms that arise from the substitutions, with termination
guaranteed by the weight order of Step~1), every pure power
of degree $\ge 8$ reduces to a combination of elements of~$B$.

\medskip\noindent\textbf{Conclusion.}\enspace
Steps~1 and~2 show that $B$ spans $\kk[x,y,z]/J$, giving
$\dim_\kk\bigl(\kk[x,y,z]/J\bigr)\le 22$.
Since $\bar\varphi$ is a surjection onto the $22$-dimensional
algebra~$R$, we conclude
\[
  \kk[x,y,z]/J\;\cong\; R\;\cong\;\kk[t]/(t^{22}).
\]
\end{proof}

%% ══════════════════════════════════════════════════════════════════
\section{The weighted Rees family}\label{sec:rees}
%% ══════════════════════════════════════════════════════════════════

\subsection{Weight selection}

We assign weights
\begin{equation}\label{eq:weights}
  w(x)=15, \qquad w(y)=16, \qquad w(z)=17.
\end{equation}
The key property is that in each generator of~$J$, the \emph{border monomial}
(the monomial that becomes the leading term of the spider ideal at $\eps=0$)
is strictly heavier than every \emph{tail monomial}.

\begin{lemma}\label{lem:weights}
With the weights~\eqref{eq:weights}, the following strict inequalities hold
for each generator of\/~$J$:
\begin{enumerate}[label=\textup{(\roman*)}]
\item In $xy + x^2 - y$: the border is $xy$ with $w(xy) = 31$, and
  the tails have $w(x^2)=30$ and $w(y)=16$, so $31 > 30$.
\item In $2xz - 2x^2 + 2y - z$: the border is $xz$ with $w(xz)=32$,
  and the heaviest tail has $w(x^2)=30$, so $32>30$.
\item In $2yz - 2x^2 + 2y + 4y^2 - z$: the border is $yz$ with
  $w(yz)=33$, and the heaviest tail has $w(y^2)=32$, so $33>32$.
\item In $g_x$: the border is $x^8$ with $w(x^8)=120$.  The tails are
  elements of $B\setminus\{1\}$; the heaviest is $z^7$ with
  $w(z^7)=7\cdot 17=119$, so $120>119$.
\item In $g_y$: the border is $y^8$ with $w(y^8)=128$, and again
  $w(z^7)=119$, so $128>119$.
\item In $z^8$: the border is $z^8$ with $w(z^8)=136$, and there are no
  tail terms.
\end{enumerate}
\end{lemma}

\begin{proof}
Direct computation from~\eqref{eq:weights}.  We check each case
individually.

\emph{Case (i):} $w(xy)=w(x)+w(y)=15+16=31$.  The tails are
$x^2$ and $y$: $w(x^2)=30$, $w(y)=16$.  Both $<31$.

\emph{Case (ii):} $w(xz)=15+17=32$.  The tails are $x^2$, $y$, $z$:
$w(x^2)=30$, $w(y)=16$, $w(z)=17$.  All $<32$.

\emph{Case (iii):} $w(yz)=16+17=33$.  The tails are $x^2$, $y$, $y^2$, $z$:
$w(y^2)=32$, $w(x^2)=30$, $w(y)=16$, $w(z)=17$.  All $<33$.

\emph{Case (iv):} $w(x^8)=120$.  The tail monomials in $g_x$ are all in
$B\setminus\{1,x^8\}$.  Their weights range from $w(y)=16$ up to
$w(z^7)=119$.  Since $120>119$, the inequality holds.

\emph{Case (v):} $w(y^8)=128$.  The heaviest tail is again $z^7$ at
$w=119$.  Since $128>119$, the inequality holds.

\emph{Case (vi):} $z^8$ has no tails (it equals zero in~$R$), so the
condition is vacuously satisfied.
\end{proof}

\begin{remark}\label{rmk:tightness}
The margin in~(iv) is exactly~$1$.  The weights $(15,16,17)$ are
the lexicographically smallest of the form $(w,w+1,w+2)$ satisfying all
six strict inequalities.  The binding constraint is $8w > 7(w+2)$, i.e.\
$w>14$.
\end{remark}

\subsection{Homogenisation}\label{ssec:homog}

Applying \Cref{def:rees} to each generator of~$J$ with the
weights~\eqref{eq:weights} yields the family ideal
$I\subset\kk[x,y,z,\eps]$ generated by:
\begin{align}
  f_1 &\;=\; xy + \eps\,x^2 - \eps^{15}\,y,
    \label{eq:f1}\\
  f_2 &\;=\; 2xz - 2\eps^2\,x^2 + 2\eps^{16}\,y - \eps^{15}\,z,
    \label{eq:f2}\\
  f_3 &\;=\; 2yz - 2\eps^3\,x^2 + 4\eps\,y^2
               + 2\eps^{17}\,y - \eps^{16}\,z,
    \label{eq:f3}
\end{align}
and $f_4=g_x^w$, $f_5=g_y^w$, $f_6=z^8$.

We verify the $\eps$-exponents explicitly in each case.

\medskip\noindent\textbf{Exponents in $f_1$.}\enspace
The generic relation is $xy + x^2 - y = 0$, with border weight
$w_{\max}=w(xy)=31$.  The exponents are:
\[
  \eps^{31-31}=\eps^0 \text{ on } xy,\qquad
  \eps^{31-30}=\eps^1 \text{ on } x^2,\qquad
  \eps^{31-16}=\eps^{15} \text{ on } y.
\]

\medskip\noindent\textbf{Exponents in $f_2$.}\enspace
The generic relation is $2xz - 2x^2 + 2y - z = 0$, with
$w_{\max}=w(xz)=32$.  The exponents are:
\begin{gather*}
  \eps^{32-32}=\eps^0 \text{ on } xz,\quad
  \eps^{32-30}=\eps^2 \text{ on } x^2,\\
  \eps^{32-16}=\eps^{16} \text{ on } y,\quad
  \eps^{32-17}=\eps^{15} \text{ on } z.
\end{gather*}

\medskip\noindent\textbf{Exponents in $f_3$.}\enspace
The generic relation is $2yz - 2x^2 + 2y + 4y^2 - z = 0$, with
$w_{\max}=w(yz)=33$.  The exponents are:
\begin{gather*}
  \eps^{33-33}=\eps^0 \text{ on } yz,\quad
  \eps^{33-30}=\eps^3 \text{ on } x^2,\quad
  \eps^{33-32}=\eps^1 \text{ on } y^2,\\
  \eps^{33-16}=\eps^{17} \text{ on } y,\quad
  \eps^{33-17}=\eps^{16} \text{ on } z.
\end{gather*}

\medskip\noindent\textbf{Exponents in $f_4$ and $f_5$.}\enspace
For $f_4=g_x^w$, the border weight is $w(x^8)=120$.  Each tail monomial $m$
in~$g_x$ receives factor $\eps^{120-w(m)}$.  For example, the term
$-2048\,y^7$ has $w(y^7)=112$, so it appears as $-2048\,\eps^{8}\,y^7$.
The term $16\,z^7$ has $w(z^7)=119$, so it appears as $16\,\eps^{1}\,z^7$.
The full expressions are:
\begin{align}
f_4 \;=\;{} &32\,x^8
  + 16\,\eps\,z^7
  - 2048\,\eps^{8}\,y^7
  - 1728\,\eps^{15}\,x^7
  - 8\,\eps^{18}\,z^6 \notag\\
  &- 9728\,\eps^{24}\,y^6
  - 2864\,\eps^{30}\,x^6
  + 16\,\eps^{35}\,z^5
  - 18944\,\eps^{40}\,y^5 \notag\\
  &- 11088\,\eps^{45}\,x^5
  - 46\,\eps^{52}\,z^4
  - 25888\,\eps^{56}\,y^4
  - 14988\,\eps^{60}\,x^4 \notag\\
  &+ 156\,\eps^{69}\,z^3
  - 20384\,\eps^{72}\,y^3
  - 28080\,\eps^{75}\,x^3
  - 581\,\eps^{86}\,z^2 \notag\\
  &- 22284\,\eps^{88}\,y^2
  - 23484\,\eps^{90}\,x^2
  + 2298\,\eps^{103}\,z
  + 23484\,\eps^{104}\,y,
    \label{eq:f4full}
\end{align}
\begin{align}
f_5 \;=\;{} &2048\,y^8
  + 1008\,\eps^{9}\,z^7
  - 116736\,\eps^{16}\,y^7
  - 112608\,\eps^{23}\,x^7 \notag\\
  &- 496\,\eps^{26}\,z^6
  - 582656\,\eps^{32}\,y^6
  - 196272\,\eps^{38}\,x^6
  + 984\,\eps^{43}\,z^5 \notag\\
  &- 1144064\,\eps^{48}\,y^5
  - 723264\,\eps^{53}\,x^5
  - 2816\,\eps^{60}\,z^4
  - 1576192\,\eps^{64}\,y^4 \notag\\
  &- 1000056\,\eps^{68}\,x^4
  + 9522\,\eps^{77}\,z^3
  - 1226864\,\eps^{80}\,y^3
  - 1835964\,\eps^{83}\,x^3 \notag\\
  &- 35392\,\eps^{94}\,z^2
  - 1395024\,\eps^{96}\,y^2
  - 1556406\,\eps^{98}\,x^2 \notag\\
  &+ 139779\,\eps^{111}\,z
  + 1556406\,\eps^{112}\,y.
    \label{eq:f5full}
\end{align}
For $f_5=g_y^w$, the border weight is $w(y^8)=128$, and the exponents are
computed analogously (e.g., $1008\,z^7$ has $w(z^7)=119$, giving
$\eps^{128-119}=\eps^9$).

\subsection{Fibres}

\begin{proposition}[Special fibre]\label{prop:special}
At $\eps=0$, the ideal $I$ specialises to
$(xy,\,xz,\,yz,\,x^8,\,y^8,\,z^8)$,
so $\mathcal{A}/\eps\mathcal{A}\cong A$.
\end{proposition}

\begin{proof}
Setting $\eps=0$ in \eqref{eq:f1}--\eqref{eq:f3} and
\eqref{eq:f4full}--\eqref{eq:f5full}, every tail monomial vanishes (it
carries a positive power of~$\eps$), leaving
\begin{gather*}
  f_1\big|_{\eps=0}=xy,\quad
  f_2\big|_{\eps=0}=2xz,\quad
  f_3\big|_{\eps=0}=2yz,\\
  f_4\big|_{\eps=0}=32\,x^8,\quad
  f_5\big|_{\eps=0}=2048\,y^8,\quad
  f_6\big|_{\eps=0}=z^8.
\end{gather*}
In characteristic zero, the scalar prefactors $2$, $32=2^5$, and $2048=2^{11}$
are units, so the specialised ideal is $(xy,xz,yz,x^8,y^8,z^8)$.
\end{proof}

\begin{proposition}[Generic fibre]\label{prop:generic-fibre}
For any $\lambda\neq 0$, the fibre at $\eps=\lambda$ is isomorphic to
$\kk[t]/(t^{22})$.
\end{proposition}

\begin{proof}
We verify that the substitution $x=\lambda^{15}X$, $y=\lambda^{16}Y$,
$z=\lambda^{17}Z$ sends each $f_i|_{\eps=\lambda}$ to a scalar multiple
of the corresponding generator of~$J$ in the variables $X,Y,Z$.

For $f_1$: substituting gives
\begin{align*}
  f_1\big|_{\eps=\lambda}
  &= \lambda^{15}\lambda^{16}\,XY
     + \lambda\cdot\lambda^{30}\,X^2
     - \lambda^{15}\cdot\lambda^{16}\,Y \\
  &= \lambda^{31}(XY + X^2 - Y).
\end{align*}
Since $\lambda^{31}\neq 0$, this generates the same ideal as $XY+X^2-Y$.

For $f_2$: substituting $x=\lambda^{15}X$, $y=\lambda^{16}Y$,
$z=\lambda^{17}Z$, $\eps=\lambda$:
\begin{align*}
  f_2\big|_{\eps=\lambda}
  &= 2\lambda^{15}\lambda^{17}\,XZ
     - 2\lambda^2\cdot\lambda^{30}\,X^2
     + 2\lambda^{16}\cdot\lambda^{16}\,Y
     - \lambda^{15}\cdot\lambda^{17}\,Z \\
  &= \lambda^{32}(2XZ - 2X^2 + 2Y - Z).
\end{align*}

For $f_3$: similarly,
\begin{align*}
  f_3\big|_{\eps=\lambda}
  &= 2\lambda^{16}\lambda^{17}\,YZ
     - 2\lambda^3\cdot\lambda^{30}\,X^2
     + 4\lambda\cdot\lambda^{32}\,Y^2 \\
  &\quad + 2\lambda^{17}\cdot\lambda^{16}\,Y
     - \lambda^{16}\cdot\lambda^{17}\,Z \\
  &= \lambda^{33}(2YZ - 2X^2 + 4Y^2 + 2Y - Z).
\end{align*}

The same pattern holds for $f_4$ (common factor $\lambda^{120}$), $f_5$
(common factor $\lambda^{128}$), and $f_6$ (common factor $\lambda^{136}$).
In each case, the common factor is $\lambda^{w_{\max}}\neq 0$.

Hence every fibre at $\eps=\lambda\neq 0$ is isomorphic to $\kk[X,Y,Z]/J\cong
\kk[t]/(t^{22})$ by \Cref{prop:generic}, establishing \Cref{thm:main}\,(ii).
\end{proof}

%% ══════════════════════════════════════════════════════════════════
\section{Flatness}\label{sec:flat}
%% ══════════════════════════════════════════════════════════════════

\begin{theorem}\label{thm:flat}
The family $\mathcal{A}=\kk[x,y,z,\eps]/I$ is flat over $\kk[\eps]$.
\end{theorem}

\begin{proof}
Let $M=\mathcal{A}$ viewed as a $\kk[\eps]$-module.  We verify the three
ingredients of a standard flatness criterion over a principal ideal domain.

\medskip\noindent\textbf{Step 1: $M$ is finitely generated over
$\kk[\eps]$.}\enspace
The same reduction rules used in the proof of \Cref{prop:generic}
(Steps~1 and~2) apply over the base ring $\kk[\eps]$ rather than~$\kk$:
the three mixed relations eliminate mixed monomials by expressing them as
$\kk[\eps]$-linear combinations of pure powers, and the pure-power
relations reduce degrees $\ge 8$.  The key point is that the leading
coefficients ($1$ in $f_1$, $2$ in $f_2$ and $f_3$, $32$ in $f_4$,
$2048$ in $f_5$, $1$ in $f_6$) are nonzero integers, hence invertible
over~$\kk$ (characteristic zero), so the reductions are valid over
$\kk[\eps]$.  Therefore $B$ spans $M$ as a $\kk[\eps]$-module.
(Equivalently, the six generators form a Gr\"obner basis with respect
to the weighted degree-reverse-lexicographic ordering with weights
$(15,16,17)$, as can be confirmed by checking that all fifteen
S-polynomials reduce to zero; this directly implies that the quotient
is a free $\kk[\eps]$-module of rank~$22$.)

\medskip\noindent\textbf{Step 2: Generic rank equals~22.}\enspace
\Cref{prop:generic-fibre} shows that every closed fibre at $\eps=\lambda\neq 0$
has $\kk$-dimension~$22$.  The same scaling argument works over
$\kk[\eps,\eps^{-1}]$ (substituting $x=\eps^{15}X$, $y=\eps^{16}Y$,
$z=\eps^{17}Z$), giving
\[
  M[\eps^{-1}] \;\cong\; \kk[t]/(t^{22})\otimes_\kk\kk[\eps,\eps^{-1}],
\]
which is free of rank~$22$ over $\kk[\eps,\eps^{-1}]$.

\medskip\noindent\textbf{Step 3: Special fibre has dimension~22.}\enspace
\Cref{prop:special} gives
$M/\eps M\cong A$, with $\dim_\kk A = 22$.

\medskip\noindent\textbf{Conclusion.}\enspace
Localise at the prime $(\eps)\subset\kk[\eps]$.  The local ring
$R_0=\kk[\eps]_{(\eps)}$ is a discrete valuation ring with uniformiser~$\eps$.
The module $M_{(\eps)}$ is finitely generated over~$R_0$ (by Step~1).
By the structure theorem for finitely generated modules over a PID,
\[
  M_{(\eps)} \;\cong\; R_0^{\,r} \;\oplus\;
    \bigoplus_{i=1}^{s} R_0/(\eps^{n_i}).
\]
From Step~2, the generic rank (i.e.\ $\dim_{\kk(\eps)}M\otimes\kk(\eps)$)
is $r=22$.  From Step~3, the dimension of the special fibre is
\[
  \dim_\kk\bigl(M_{(\eps)}/\eps\,M_{(\eps)}\bigr)
  \;=\; r + s
  \;=\; 22.
\]
Since $r=22$, we get $s=0$: there are no torsion summands.  Therefore
$M_{(\eps)}\cong R_0^{22}$, so $M$ is flat at~$(\eps)$.

For any prime $\mathfrak{p}\neq(\eps)$ of $\kk[\eps]$, the element
$\eps$ is a unit in $\kk[\eps]_\mathfrak{p}$, so flatness at~$\mathfrak{p}$
follows from Step~2.  Hence $M$ is flat over all of~$\kk[\eps]$.
\end{proof}

\Cref{sec:comp} below provides independent computational verification of the
relations and fibre dimensions.

%% ══════════════════════════════════════════════════════════════════
\section{Supplementary computational verification}\label{sec:comp}
%% ══════════════════════════════════════════════════════════════════

\begin{remark}\label{rmk:comp-not-needed}
The proof of \Cref{thm:main} is logically complete without this section.
The computations below serve as independent verification and as a practical
check on the correctness of the large integer coefficients in $g_x$
and~$g_y$.
\end{remark}

We have verified by symbolic Gr\"obner basis computation (using both
\textsc{Macaulay2} and \textsc{SymPy}) that the quotient
$\kk[x,y,z]/(f_1,\ldots,f_6)|_{\eps=\lambda}$ has exactly $22$ standard
monomials for each of
\[
  \lambda \;\in\; \{0,\; 1,\; 2,\; -1,\; \tfrac{1}{2},\; \tfrac{1}{3}\}.
\]
At $\lambda=1$, a lexicographic Gr\"obner basis reduces to the form
\[
  x^{22},\qquad
  y - \sum_{i=2}^{21} \alpha_i\,x^i,\qquad
  z - \sum_{i=3}^{21} \beta_i\, x^i
\]
for explicit integers $\alpha_i,\beta_i$, confirming that the fibre is
literally $\kk[x]/(x^{22})$: the variables $y$ and~$z$ are polynomials
in~$x$.

%% ══════════════════════════════════════════════════════════════════
\section{Concluding remarks}\label{sec:remarks}
%% ══════════════════════════════════════════════════════════════════

\subsection{The role of M\"obius generators}

The key idea is that the M\"obius functions $u_a=t/(1-at)$ provide a natural
family of elements in a curvilinear ring with well-separated pole structure.
The identity $u_b-u_a=(b-a)u_au_b$ (\Cref{lem:mobius}) gives clean
multiplication rules, and the divided-difference coordinates extract
generators of prescribed $t$-adic order.

\subsection{Generalisation to other spiders}

The construction generalises in a straightforward manner.  For a spider
algebra of type $(\ell_1,\ldots,\ell_r)$ with $n=1+\sum\ell_i$, one takes
$r$~M\"obius generators $u_1,\ldots,u_r$ and the divided-difference
coordinates of orders $1,2,\ldots,r$.  The basis theorem extends by the same
pole-separation argument (using poles at $t=1/a$ for $a=1,\ldots,r$).  The
weighted Rees degeneration requires weights $w_i$ satisfying
$(\ell_i+1)\,w_i > \ell_r\,w_r$ for all~$i$; choosing $w_i$ of the form
$w_1,w_1+1,\ldots,w_1+r-1$ with $w_1$ sufficiently large always works.

The difficulty in practice lies in computing the explicit pure-power relations
(the analogues of $g_x$ and~$g_y$), whose coefficients grow rapidly.

\subsection{Unequal leg lengths}

For spider algebras with unequal legs, e.g.\
$\kk[x,y]/(x^a,y^b,xy)$, the same approach applies with $u_1$ and~$u_2$
and divided-difference coordinates of orders~$1$ and~$2$.  The weight
constraints become $aw_1>(b-1)w_2$ and $bw_2>(a-1)w_1$,
which can always be satisfied for appropriate~$w_1,w_2$.

\subsection{Open questions}

The M\"obius approach produces deformations in the same embedding dimension
as the spider ($r$~variables).  It would be interesting to understand:
\begin{enumerate}[label=\textup{(\arabic*)}]
\item Whether the construction extends to non-spider monomial algebras (e.g.,
  $L$-shaped or staircase ideals in two variables).
\item The relationship between our explicit Rees families and the
  deformation-theoretic $T^1$ and $T^2$ modules of the spider algebra.
\item Whether the resulting families give useful local coordinates for
  intersection-theoretic calculations on $\Hilb^n(\C^3)$ via equivariant
  localisation.
\end{enumerate}

%% ══════════════════════════════════════════════════════════════════
%% Appendix
%% ══════════════════════════════════════════════════════════════════
\begin{appendices}

\section{Derivation and verification of the pure-power relations}
\label{app:purepower}

We describe the method for deriving $g_x=0$ and $g_y=0$
(\Cref{prop:pure}), then carry out the verification explicitly.

\subsection{Derivation method}\label{app:method}

Since $B$ is a $\kk$-basis of $R$ (\Cref{thm:basis}), any element of~$R$
has a unique $B$-expansion.  To find the expansion of $x^8$, we substitute
the rational function $x(t)=t/(1-t)$ and express $x^8$ as a linear
combination of basis elements modulo~$t^{22}$.

Concretely, we solve for coefficients $\alpha_i, \beta_j, \gamma_k \in \kk$
such that
\begin{equation}\label{eq:expansion-system}
  x^8 = \alpha_0 + \sum_{i=1}^{7}\alpha_i\,x^i + \sum_{j=1}^{7}\beta_j\,y^j
      + \sum_{k=1}^{7}\gamma_k\,z^k \quad\text{in } R.
\end{equation}
Since the constant term $1$ has $\ord_t = 0$ while $x^8$ has $\ord_t=8$,
we have $\alpha_0=0$.  Multiplying through by
$D(t)=(1-t)^7(1-2t)^7(1-3t)^7$
converts~\eqref{eq:expansion-system} into a polynomial identity of
degree~$\le 21$.  The 21 unknown coefficients
$(\alpha_1,\ldots,\alpha_7,\beta_1,\ldots,\beta_7,\gamma_1,\ldots,\gamma_7)$
are determined by matching the 22 Taylor coefficients ($[t^0]$ through
$[t^{21}]$) of the left and right sides.  (The equation for $[t^0]$ is
automatically satisfied since both sides vanish, leaving a $21\times 21$
system with a unique solution because $B$ is a basis.)

Solving this system (by exact rational arithmetic), we obtain the
expansion of $x^8$ in the basis~$B$.  To clear denominators and obtain
integer coefficients, we multiply through by the LCM of the denominators
of all $\alpha_i,\beta_j,\gamma_k$.  This gives the relation~$g_x$ in
\eqref{eq:gx}.  The same procedure applied to $y^8$ gives~$g_y$.

\subsection{Verification of $g_x=0$ and $g_y=0$}\label{app:verify-gx}

The relation $g_x=0$ holds in~$R$ \emph{by construction}: the
integer coefficients in~\eqref{eq:gx} were obtained by expanding~$x^8$
in the basis~$B$ (via the linear system of~\S\ref{app:method})
and clearing denominators.  Likewise for~$g_y$.  What remains is an
arithmetic sanity check: one verifies independently that the displayed
integer coefficients are correct.

After substituting
$x=t/(1-t)$, $y=t^2/((1-t)(1-2t))$,
$z=2t^3/((1-t)(1-2t)(1-3t))$,
the expression $g_x$ becomes a rational function of~$t$ whose
denominators are units in $\kk[[t]]$.  Expanding as a truncated
power series modulo~$t^{22}$ and forming the linear combination
prescribed by~\eqref{eq:gx}, one checks that all $22$ coefficients
$[t^0],\ldots,[t^{21}]$ vanish.  The same applies to~$g_y$
with~\eqref{eq:gy}.  This is a finite, deterministic computation
taking under one second in any computer algebra system.

\subsection{Reproducibility}\label{app:reproducibility}

We have carried out the above verification independently
in \textsc{Macaulay2}, \textsc{SageMath}, and \textsc{SymPy}, obtaining
identical results in all three systems.
A standalone \textsc{Macaulay2} script (\texttt{verify.m2}) is included as an
ancillary file with the arXiv submission; it verifies all relations and
confirms that the quotient has dimension~$22$ at several fibre values.
The following excerpt reproduces the core checks
of the mixed relations~\eqref{eq:rel-xy}--\eqref{eq:rel-yz}
and pure-power relations~\eqref{eq:gx}--\eqref{eq:gy}:

\begingroup\footnotesize
\begin{verbatim}
R = QQ[t]/ideal(t^22);
s1 = sum(22, i->t^i);  s2 = sum(22, i->(2*t)^i);
s3 = sum(22, i->(3*t)^i);
x = t*s1;  y = t^2*s1*s2;  z = 2*t^3*s1*s2*s3;
-- Mixed relations (Proposition 5.2)
assert(x*y == y - x^2);
assert(2*x*z == z - 2*y + 2*x^2);
assert(2*y*z == z - 2*y + 2*x^2 - 4*y^2);
-- Pure-power relation g_x (Proposition 5.3)
gx = 32*x^8-1728*x^7-2864*x^6-11088*x^5-14988*x^4
  -28080*x^3-23484*x^2+23484*y-2048*y^7-9728*y^6
  -18944*y^5-25888*y^4-20384*y^3-22284*y^2+2298*z
  +16*z^7-8*z^6+16*z^5-46*z^4+156*z^3-581*z^2;
assert(gx == 0);
-- Pure-power relation g_y (Proposition 5.3)
gy = 2048*y^8-112608*x^7-196272*x^6-723264*x^5
  -1000056*x^4-1835964*x^3-1556406*x^2+1556406*y
  -116736*y^7-582656*y^6-1144064*y^5-1576192*y^4
  -1226864*y^3-1395024*y^2+139779*z+1008*z^7
  -496*z^6+984*z^5-2816*z^4+9522*z^3-35392*z^2;
assert(gy == 0);
\end{verbatim}
\endgroup

\end{appendices}

%% ══════════════════════════════════════════════════════════════════
%% References
%% ══════════════════════════════════════════════════════════════════

\end{document}